\documentclass[11pt]{article}
\usepackage[top=2.5cm, bottom=2.5cm, left=2.5cm, right=2.5cm]{geometry}
\usepackage{amsthm}
\usepackage{amsmath}
\usepackage{amssymb}
\usepackage{amsfonts}
\usepackage{latexsym}
\usepackage{mathtools}
\usepackage{enumitem}

\usepackage[pdftex]{graphicx} 
\usepackage[pdftex]{xcolor}   
\usepackage{bm}
\usepackage{ytableau}
\usepackage{tikz}
\usepackage{comment}

\newcommand{\N}{\mathbb{N}}

\newcommand{\Z}{\mathbb{Z}}

\newcommand{\C}{\mathbb{C}}

\newcommand{\SSYT}{{\rm SSYT}}
\newcommand{\PSST}{{\rm PSST}}
\newcommand{\QSST}{{\rm QSST}}
\newcommand{\Sym}{{\rm Sym}}
\newcommand{\Rect}{{\rm Rect}}
\newcommand{\shape}{{\rm shape}}

\newtheorem{thms}{Theorem}[section]
\newtheorem*{thms*}{Theorem}
\newtheorem{lem}[thms]{Lemma}
\newtheorem{cor}[thms]{Corollary}
\newtheorem{prop}[thms]{Proposition}

\newtheorem*{fact*}{Fact}

\theoremstyle{definition}
\newtheorem{dfn}[thms]{Definition}
\newtheorem{ex}[thms]{Example}
\newtheorem*{ex*}{Example}
\newtheorem{rmk}[thms]{Remark}

\author{Hikari Hanaki}
\date{}

\begin{document}

\title{The Littlewood-Richardson rule for Schur $P$-, $Q$-multiple zeta functions}

\maketitle

\begin{abstract}
The Schur $P$-, $Q$-multiple zeta functions were defined by Nakasuji and Takeda inspired by the tableau representation of Schur $P$-, $Q$-functions. 
While a product of two Schur $P$-functions expands as a linear combination of Schur $P$-functions, we obtain a similar expansion formula for the Schur $P$-multiple zeta functions by taking summation over the symmetric group permutating all the variables. We also introduce a expansion formula of skew Schur $Q$-multiple zeta functions by taking summation over the symmetric group. Furthermore, this skew type formula can be refined by restricting the symmetric group to its specific subgroup. 
\end{abstract}

\section{Introduction}

The {\it Schur $P$-, $Q$-multiple zeta functions} were introduced by Nakasuji-Takeda \cite{NT}. 
These are defined as sums over combinatorial objects called semi-standard marked shifted tableau, similar to usual Schur $P$-, $Q$-functions. 
We now review the detailed definitions. 

A {\it partition} is a finite sequence $\lambda=(\lambda_1, \dots, \lambda_r)$ of non-negative integers in weakly decreasing order: $\lambda_1 \ge \dots \ge \lambda_r$. The number of non-zero $\lambda_i$ in $\lambda=(\lambda_1, \dots, \lambda_r)$ is the {\it length} of $\lambda$, denoted by $l(\lambda)$, and the sum of $\lambda_i$ is the {\it weight} of $\lambda$, denoted by $|\lambda|$. If $|\lambda|=n$, we say that $\lambda$ is a partition of $n$. 
A partition $\lambda=(\lambda_1,\ldots,\lambda_r)$ is called {\it strict}, if $\lambda_1 > \cdots > \lambda_r \ge 0$. 
Throughout this paper, let $\lambda, \mu, \nu$ be strict partitions. 
For $\lambda = (\lambda_1, \dots, \lambda_r)$, let $SD(\lambda)$ be the subset of $\Z^2$:
$$SD(\lambda)=\{(i,j)\in\Z^2|1\le i\le r, i \le j\le \lambda_i + i - 1\}$$
and depicted as a collection of square boxes arranged in the $i$-th row having $\lambda_i$ boxes and indented $i-1$ spaces. We call this the {\it shifted diagram} of shape $\lambda$. 
We identify $\lambda$ with this collection of boxes, and we call it {\it shifted shape} $\lambda$ and denote it by the same symbol $\lambda$. 
We say that $(i,j)\in SD(\lambda)$ is a {\it corner} of $\lambda$ if $(i+1,j)\notin SD(\lambda)$ and $(i,j+1)\notin SD(\lambda)$ and denote by $SC(\lambda)\subset SD(\lambda)$ the set of all corners of $\lambda$.

A {\it skew strict partition} is a pair of strict partitions $\lambda/\mu$ such that $\mu_i \le \lambda_i$ for all $i$. The shifted diagram of the skew strict partition $\lambda/\mu$ denoted by $SD(\lambda/\mu)$, is defined by $SD(\lambda/\mu) = SD(\lambda) \setminus SD(\mu)$ and is depicted as the diagram of removing that of $\mu$ from that of $\lambda$. For clarity, we depict the diagram of $\lambda/\mu$ by coloring the boxes of $\mu$ in black as in Figure \ref{skew_fig}. 
When $\mu = \emptyset$ (i.e., $\mu_i = 0$ for all $i$), we identify the skew strict partition $\lambda/\mu$ with $\lambda$ and call it {\it the normal strict partition}. Unless otherwise stated, partitions are assumed to be non-empty.  The corners of $\lambda/\mu$ are defined in the same way as in the case of normal partition. 

\begin{figure}[h]
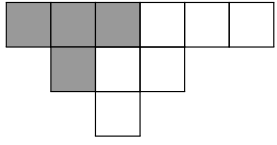

\center
\begin{ytableau}
*(black!40){} & *(black!40){} & *(black!40){} & {} & {} & {}\\
\none & *(black!40){} & {} & {} \\
\none & \none & {}
\end{ytableau}
\caption{Example of diagram of skew strict partition: $\lambda/\mu = (6, 3, 1)/(3, 1)$}
\label{skew_fig}
\end{figure}

Let $X$ be a set. A {\it shifted tableau} $T=(t_{ij})$ of shifted shape $\lambda$ over $X$ is a filling of $SD(\lambda)$ obtained by putting $t_{ij}\in X$ into the $(i,j)$ box of $SD(\lambda)$. We denote by $ST(\lambda, X)$ the set of all shifted tableaux of shape $\lambda$ over $X$. 
A {\it skew shifted tableau} is defined analogously.
We denote the shape of a shifted tableau $T$ by $\shape(T)$.

Let $\mathbb N'$ be the set $\{1',1,2',2,\ldots\}$ with the total ordering $1' < 1 < 2' < 2 <\cdots $. The symbols $1', 2', \dots$ are said to be {\it marked}, and we denote $|k| = |k'| = k$ for each $k \in \N$. For a normal or skew shifted tableau $T = (t_{ij})$ over $\N$, we denote the tableau $(|t_{ij}|)$ by $|T|$. Similarly, for a word $w$ over $\N$, we denote the word obtained by removing all marks from $w$ by $|w|$. The {\it content} of a tableau $T$ over $\N'$ is the sequence $(n_1, \dots, n_r)$, where $n_k$ is the number of times the entries $k'$ and $k$ appear in $T$. 

Let $\delta = \lambda/\mu$ (now, including $\mu = \emptyset$). 
Then, a {\it semi-standard marked shifted tableau} $M=(m_{ij})\in ST(\delta,\mathbb N')$ is a shifted tableau satisfying the following conditions. 
\begin{description}
    \item[PST1] The entries of $M$ are weakly increasing along each column and row of $M$. 
    \item[PST2] For each $k=1,2,\ldots$, there is at most one $k'$ per row.
    \item[PST3] For each $k=1,2,\ldots$, there is at most one $k$ per column.
    \item[PST4] There is no $k'$ on the main diagonal.
\end{description}
We denote by $\PSST(\delta)$ the set of semi-standard marked shifted tableaux of shape $\delta$.
Similarly, we denote by $\QSST(\delta)$ the set of the shifted tableaux of shape $\delta$ over $\N'$ with PST1, PST2 and PST3.
For short, we denote the union of $\PSST(\delta)$ for all normal or all skew shifted shapes $\delta$ by $\PSST$ and denote that of $\QSST(\delta)$ for all normal or all skew shifted shapes $\delta$ by $\QSST$.

For a given shifted tableau $\bm{v}=(v_{ij})\in ST(\delta, \C)$ of variables, 
the {\it Schur $P$-multiple zeta function} and the {\it Schur $Q$-multiple zeta function} of shape $\delta$ are defined as
$$\zeta_\delta^P({\bm v})=\sum_{M\in \PSST(\delta)}\frac{1}{M^{\bm v}},$$
and
$$\zeta_\delta^Q({\bm v})=\sum_{M\in \QSST(\delta)}\frac{1}{M^{\bm v}},$$
respectively, where $M^{\bm{v}}=\displaystyle{\prod_{(i, j)\in SD(\delta)}|m_{ij}|^{v_{ij}}}$ for $M=(m_{ij})\in \QSST(\delta)$.
Both functions $\zeta_\delta^P(\bm{v})$ and $\zeta_\delta^Q(\bm{v})$ converge absolutely in 
$$
 W_{\delta}^Q
=
\left\{\bm{v}=(v_{ij})\in ST(\delta,\mathbb{C})\,\left|\,
\begin{array}{l}
 \text{$\Re(v_{ij})\ge 1$ for all $(i,j)\in SD(\delta) \setminus SC(\delta)$}, \\[3pt]
 \text{$\Re(v_{ij})>1$ for all $(i,j)\in SC(\delta)$}
\end{array}
\right.
\right\}.
$$

{\it Schur P-, Q-functions} are symmetric functions expressed in terms of shifted tableau as follows: for a normal or skew strict partition $\delta$,
$$P_\delta = \sum_{(m_{ij}) \in \PSST(\delta)} \prod_{(i,j) \in SD(\delta)} x_{m_{ij}},$$
and
$$Q_\delta = \sum_{(m_{ij}) \in \QSST(\delta)} \prod_{(i,j) \in SD(\delta)} x_{m_{ij}}.$$
Since the Schur $P$-, $Q$-multiple zeta functions have a structure similar to that of the Schur $P$-, $Q$-functions, they are expected to have similar properties to those for Schur $P$-, $Q$-functions. 
It is known that Schur $P$-functions have the expansion
\begin{align}\label{LR1}
P_\mu P_\nu=\sum_{\lambda: \rm{strict\ partition}}f_{\mu\nu}^\lambda P_\lambda .
\end{align}
It is known that the coefficients $f_{\mu\nu}^\lambda$ satisfying the equality (\ref{LR1}) are unique for $\mu, \nu, \lambda$, and then $f_{\mu\nu}^\lambda$ are defined by this expansion and sometimes called {\it shifted Littlewood-Richardson coefficients} (cf. \cite{Ser}). 
It also holds that 
\begin{align}\label{LR2}
Q_{\lambda/\mu}=\sum_{\nu:\text{strict\ partition}}f_{\mu\nu}^\lambda Q_\nu 
\end{align}
for Schur $Q$-functions.
There are some combinatorial rules for computing the shifted Littlewood-Richardson coefficients. 
Theorem \ref{LRrule2} is one of the rules, and it implies that $f_{\mu\nu}^\lambda \in \Z_{\ge 0}$.   
Therefore, for $\lambda, \mu, \nu$ such that $f_{\mu\nu}^\lambda > 0$, we have that $|\mu| + |\nu| = |\lambda|$ by comparing the degree of the two sides of equality (1). 
We define the set $\mathcal{G}^P(\mu, \nu)$ as the set of all strict partitions $\lambda$ such that $f_{\mu\nu}^\lambda > 0$, and define the set $\mathcal{G}^Q(\lambda/\mu)$ as the set of all partitions $\nu$ such that $f_{\mu\nu}^\lambda > 0$. 

In this paper, we introduce analogous equality to (\ref{LR1}) for Schur $P$-multiple zeta functions and analogous equality to (\ref{LR2}) for Schur $Q$-multiple zeta functions.

\begin{thms}\label{main}

Let $\mu, \nu$ be strict partitions. 
Let $\bm{s}=(s_{ij})\in ST(\mu, \mathbb{C}), \bm{t}=(t_{ij})\in ST(\nu,\mathbb{C})$ be variables. 
Assume that the real parts of all variables $\{s_{ij}\}, \{t_{ij}\}$ are greater than 1.
For $\lambda \in \mathcal{G}^P(\mu,\nu)$, let $\bm{u}_{\lambda}(\bm{s}, \bm{t}) \in ST(\lambda, \C)$ be a tableau of shifted shape $\lambda$ whose entries are the variables from the set $\{s_{ij}\} \cup \{t_{ij}\}$, with each variable appearing exactly once.
Then the following equality holds for any such $\bm{u}_{\lambda}(\bm{s}, \bm{t})$:
\begin{align*}
\sum_{\Sym(\bm{s}, \bm{t})}\zeta_{\mu}^P(\bm{s})\zeta_{\nu}^P(\bm{t})
=\sum_{\Sym(\bm{s}, \bm{t})}\sum_{\lambda\in\mathcal{G}^P(\mu,\nu)}
f_{\mu\nu}^{\lambda}\zeta_{\lambda}^P(\bm{u}_{\lambda}(\bm{s},\bm{t}))
\end{align*}
where the $\sum_{\Sym(\bm{s}, \bm{t})}$ is the summation over the symmetric group permuting all the variables in $\{s_{ij}\} \cup \{t_{ij}\}$. 
\end{thms}

For the expansion of skew Schur $Q$-multiple zeta function $\zeta_{\lambda/\mu}^Q(\bm{v})$, we obtain an analogue of (\ref{LR2}) similarly by symmetrizing (i.e., taking the summation over the symmetric group on all the variables $v_{ij}$ in $\bm{v}$). We denote this symmetric group $\Sym(\bm{v})$. 
Furthermore, the analogue of (\ref{LR2}) can be refined by replacing the symmetric group $\Sym(\bm{v})$ with its subgroup $\Sym(SB(\bm{v}))$.

\begin{thms}\label{skew}
Let $\lambda/\mu$ be a skew strict partition. Let $\bm{v} = (v_{ij}) \in ST(\lambda/\mu, \C)$ be variables. 
Assume that the real parts of the variables $v_{1j}$ with ${\rm min}\{\lambda_2, \sum_{i \ge 2}(\lambda_i - \mu_i)\} + \mu_1 + 1 \le j \le \lambda_1 - 1$  
are greater than or equal to 1 and the real parts of other variables are greater than 1.
Then the following equality holds for any $(\bm{u}_\nu(\bm{v}) \in U_{\nu}^Q(\bm{v}))_{\nu \in \mathcal{G}^Q(\lambda/\mu)}$:
\begin{align*}
\sum_{\Sym(SB(\bm{v}))} \zeta_{\lambda/\mu}^Q(\bm{v})
=\sum_{\Sym(SB(\bm{v}))} \sum_{\nu \in \mathcal{G}^Q(\lambda/\mu)} 
f_{\mu\nu}^{\lambda}\zeta_{\nu}^Q(\bm{u}_{\nu}(\bm{v}))
\end{align*}
where the $\sum_{\Sym(SB(\bm{v}))}$ is the summation over the symmetric group permuting the variables $\{v_{ij}\}$ in $SB(\bm{v})$. 
\end{thms}

Remark \ref{rmk_body} and Remark \ref{rmk_U} provide initial explanation of set $SB(\bm{v})$ and $U_\nu^Q(\bm{v})$; the formal definition, however, will be provided in \S \ref{prep_subs} (Definition \ref{arm_body_def}) and \S \ref{proof_skew_subs} respectively.

\begin{rmk}\label{rmk_body}
For $\bm{v} \in \QSST(\lambda/\mu)$, the tableau $SB(\bm{v})$ called the {\it shifted body} of $\bm{v}$ is obtained by removing the {\it shifted arm} of $\bm{v}$ (denoted by $SA(\bm{v})$), which is a part of $\bm{v}$ consisting of a certain number of contiguous boxes with elements from the right-most box in the first row to left.
For example, for $\lambda = (8, 3, 1)$, $\mu = (3, 1)$ and $\bm{v} = (v_{ij}) \in ST(\lambda/\mu, \C)$, the part of light gray boxes in the following figure is $SB(\bm{v})$, and the part of white boxes is $SA(\bm{v})$.
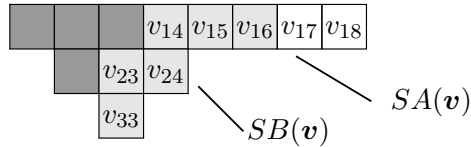
$$
\vcenter{\hbox{
\begin{tikzpicture}[baseline=(current bounding box.center)]
\node at (0,0){ 
~\begin{ytableau}
*(black!40){} & *(black!40){} & *(black!40){} & *(black!10)v_{14} & *(black!10)v_{15} & *(black!10)v_{16} & v_{17} & v_{18}\\
\none & *(black!40){} & *(black!10)v_{23} & *(black!10)v_{24} \\
\none & \none & *(black!10)v_{33}
\end{ytableau}~};
\draw[line width=0.7pt] (1.5, 0.2) -- (2.5, -0.4) node[at={(3.3, -0.4)}]{$SA(\bm{v})$};
\draw[line width=0.7pt] (0.2, -0.2) -- (0.8, -0.8) node[at={(1.4, -0.8)}]{$SB(\bm{v})$};
\end{tikzpicture}}}
$$
\end{rmk}

\begin{rmk}\label{rmk_U}
The set $U_\nu^Q(\bm{v})$ for $\nu \in \mathcal{G}^Q(\lambda/\mu)$ consists of shifted tableaux of shape $\nu$ satisfying the following conditions.
\begin{itemize}
\item The entries are the variables from the set $\{v_{ij}\}$, with each variables appearing exactly once.
\item Having $SA(\bm{v})$ in the right-most part of first row.
\end{itemize}

For example, for $\lambda/\mu = (8, 3, 1)/(3, 1)$, $\nu = (6, 2)$, and 
$$
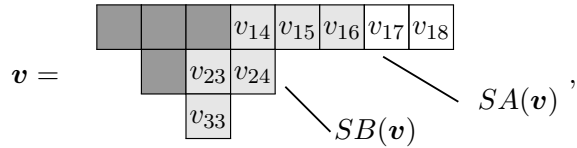
\bm{v} = 
\vcenter{\hbox{
\begin{tikzpicture}[baseline=(current bounding box.center)]
\node at (0,0){ 
~\begin{ytableau}
*(black!40){} & *(black!40){} & *(black!40){} & *(black!10)v_{14} & *(black!10)v_{15} & *(black!10)v_{16} & v_{17} & v_{18}\\
\none & *(black!40){} & *(black!10)v_{23} & *(black!10)v_{24} \\
\none & \none & *(black!10)v_{33}
\end{ytableau}~};
\draw[line width=0.7pt] (1.5, 0.2) -- (2.5, -0.4) node[at={(3.3, -0.4)}]{$SA(\bm{v})$};
\draw[line width=0.7pt] (0.2, -0.2) -- (0.8, -0.8) node[at={(1.4, -0.8)}]{$SB(\bm{v})$};
\end{tikzpicture}}}, $$
the following tableau is an example of $\bm{u}_\nu(\bm{v}) \in U_\nu^Q(\bm{v})$:
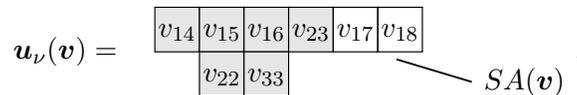
$$
\bm{u}_\nu(\bm{v}) = 
\vcenter{\hbox{
\begin{tikzpicture}[baseline=(current bounding box.center)]
\node at (0,0){ 
~\begin{ytableau}
*(black!10)v_{14} & *(black!10)v_{15} & *(black!10)v_{16} & *(black!10)v_{23} & v_{17} & v_{18}\\
\none & *(black!10)v_{22} & *(black!10)v_{33}
\end{ytableau}~};
\draw[line width=0.7pt] (1.5, -0.1) -- (2.5, -0.4) node[at={(3.2, -0.4)}]{$SA(\bm{v})$};
\end{tikzpicture}}}.
$$

\end{rmk}

This paper is organized as follows. In \S2, we give a proof of Theorem \ref{main}.
In \S3, we review Knuth equivalence and jeu de taquin, which are the combinatorial background of semi-standard shifted tableau and give a proof of Theorem \ref{skew}.

\newpage

\section{Proof of Theorem \ref{main}}
In this section, we provide a proof of Theorem \ref{main}.

\begin{proof}
For a shifted tableau $M$ over $\N'$, let $\bm{n}(M) = (n_k(M))_{k \in \N}$ be the sequence over $\N$ such that $(n_1(M), \dots, n_r(M))$ is the content of $M$ for all sufficiently large $r \in \N$.
For a sequence $\bm{n} = (n_k)_{k \in \N}$ over $\N$, let $\mathcal{M}_{\mu\nu}^P(\bm{n})$ be the set of $(M_0, M_1) \in \PSST(\mu) \times \PSST(\nu)$ such that $n_k(M_0) + n_k(M_1) = n_k$ for all $k \in \N$. Moreover, for $\bm{n}$, let $\mathcal{M}_\lambda^P(\bm{n})$ be the set of $L \in \PSST$ such that $\bm{n}(L) = \bm{n}$.

We now consider the monomial expansion of the product $P_\mu P_\nu$:
\begin{align*}
P_\mu P_\nu
 &= \sum_{\bm{n} = (n_k)_{k \in \N}} \sum_{(M_0, M_1) \in \mathcal{M}_{\mu\nu}^P(\bm{n})} \prod_{k \in \N} x_k^{n_k(M_0)} x_k^{n_k(M_1)}\\
 &= \sum_{\bm{n} = (n_k)_{k \in \N}} \sum_{(M_0, M_1) \in \mathcal{M}_{\mu\nu}^P(\bm{n})} \prod_{k \in \N}x_k^{n_k} \\
 &= \sum_{\bm{n} = (n_k)_{k \in \N}} \#\mathcal{M}_{\mu\nu}^P(\bm{n}) \prod_{k \in \N}  x_k^{n_k} 
\end{align*}
where $\bm{n}$ runs over all sequences over $\N$.
By the expansion (\ref{LR1}), we also obtain the following expansion:
\begin{align*}
P_\mu P_\nu
& = \sum_{\lambda \in \mathcal{G}^P(\mu, \nu)} f_{\mu\nu}^\lambda P_{\lambda}\\
& = \sum _{\bm{n} = (n_k)_{k \in \N}} \sum_{\lambda \in \mathcal{G}^P(\mu, \nu)} f_{\mu\nu}^\lambda \sum_{L \in \mathcal{M}_\lambda^P(\bm{n})} \prod_{k \in \N} x_k^{n_k(L)} \\
& = \sum_{\bm{n} = (n_k)_{k \in \N}} \sum_{\lambda \in \mathcal{G}^P(\mu, \nu)} f_{\mu\nu}^\lambda \sum_{L \in \mathcal{M}_\lambda^P(\bm{n})} \prod_{k \in \N} x_k^{n_k}\\
& = \sum_{\bm{n} = (n_k)_{k \in \N}} \sum_{\lambda \in \mathcal{G}^P(\mu, \nu)} f_{\mu\nu}^\lambda \#\mathcal{M}_\lambda^P(\bm{n}) \prod_{k \in \N} x_k^{n_k}
\end{align*}
where $\bm{n}$ runs over all sequences over $\N$.

For each $\bm{n}$, by comparing the coefficients of the monomial $\displaystyle{\prod_{k \in \N}x_k^{n_k}}$ of the above calculations, we have 
$$\#\mathcal{M}_{\mu\nu}^P (\bm{n}) = \sum_{\lambda \in \mathcal{G}^P(\mu, \nu)} f_{\mu\nu}^\lambda \#\mathcal{M}_\lambda^P(\bm{n}).$$

For $\lambda \in \mathcal{G}^P(\mu,\nu)$, let $\bm{u}_{\lambda}(\bm{s}, \bm{t}) \in ST(\lambda, \C)$ be a tableau of shifted shape $\lambda$ whose entries are the variables from the set $\{s_{ij}\} \cup \{t_{ij}\}$, with each variable appearing exactly once.
For a sequence $\bm{n}$ over $\N$, the functions
$M_0^{\bm{s}}M_1^{\bm{t}}$ and $L^{\bm{u}_{\lambda}(\bm{s}, \bm{t})}$ can be mapped to each other by the action $\Sym(\bm{s}, \bm{t})$ for any $(M_0, M_1) \in \mathcal{M}_{\mu\nu}^P(\bm{n})$ and $L \in \mathcal{M}_\lambda^P(\bm{n})$. 
In particular, for a fixed $L_0 \in \mathcal{M}_\lambda^P(\bm{n})$, the functions $M_0^{\bm{s}} M_1^{\bm{t}}$ and $L^{\bm{u}_{\lambda}(\bm{s}, \bm{t})}$ can be mapped to $L_0^{\bm{u}_{\lambda}(\bm{s}, \bm{t})}$ by the action $\Sym(\bm{s}, \bm{t})$ for any $(M_0, M_1) \in \mathcal{M}_{\mu\nu}^P(\bm{n})$ and $L \in \mathcal{M}_\lambda^P(\bm{n})$.
From the discussion above, we have the following calculation:
\begin{align*}
\sum_{\Sym(\bm{s}, \bm{t})} \zeta_{\mu}^P(\bm{s}) \zeta_\nu^P(\bm{t})
& = \sum_{\bm{n}} \sum_{(M_0, M_1)\in \mathcal{M}_{\mu\nu}^P(\bm{n})} \sum_{\Sym(\bm{s}, \bm{t})} \frac{1}{M_0^{\bm{s}} M_1^{\bm{t}}}\\
& = \sum_{\bm{n}} \sum_{(M_0, M_1)\in \mathcal{M}_{\mu\nu}^P(\bm{n})} \sum_{\Sym(\bm{s}, \bm{t})} \frac{1}{L_0^{\bm{u}_\lambda(\bm{s}, \bm{t})}}\\
& = \sum_{\bm{n}} \sum_{\Sym(\bm{s}, \bm{t})} \#\mathcal{M}_{\mu\nu}^P(\bm{n}) \frac{1}{L_0^{\bm{u}_\lambda(\bm{s}, \bm{t})}}\\
& = \sum_{\bm{n}} \sum_{\Sym(\bm{s}, \bm{t})} \sum_{\lambda \in \mathcal{G}^P(\mu, \nu)} f_{\mu\nu}^\lambda \#\mathcal{M}_{\lambda}^P(\bm{n}) \frac{1}{L_0^{\bm{u}_\lambda(\bm{s}, \bm{t})}}\\
& = \sum_{\bm{n}} \sum_{\Sym(\bm{s}, \bm{t})} \sum_{\lambda \in \mathcal{G}^P(\mu, \nu)} \sum_{L \in \mathcal{M}_{\lambda}^P(\bm{n})} f_{\mu\nu}^\lambda \frac{1}{L^{\bm{u}_\lambda(\bm{s}, \bm{t})}}\\
& = \sum_{\Sym(\bm{s}, \bm{t})} \sum_{\lambda \in \mathcal{G}^P(\mu, \nu)} f_{\mu\nu}^\lambda \zeta_\lambda^P(\bm{u}_\lambda(\bm{s}, \bm{t}))
\end{align*}
where $\bm{n}$ runs over all sequences over $\N$. 
Thus, we obtain Theorem \ref{main}.
\end{proof}

\section{Proof of Theorem \ref{skew}}

In \S3.1, we introduce the combinatorial background for shifted tableau of $\QSST$: Knuth equivalence and jeu de taquin. 
\S3.2 is devoted to providing the setup for the proof of Theorem \ref{skew}, which is presented in \S3.3.

\subsection{Combinatorial background}
This subsection is devoted to a review of Knuth equivalence and jeu de taquin, which are the combinatorial background for shifted tableaux. These combinatorial theories of shifted tableau were established by Worley \cite{W} and Sagan \cite{Sag}.

\subsubsection{Knuth equivalence}\label{Knuth}
We write a words as sequence of letters and write $w^{(0)} \cdot w^{(1)}$ or $w^{(0)}w^{(1)}$ for the word which is the juxtaposition of the two words $w^{(0)}$ and $w^{(1)}$.
The {\it content} of a word over $\N'$ is the sequence $(n_1, \dots, n_r)$ such that $n_k$ is the number of times the occurrences $k'$ and $k$ appear in $w$.
The {\it row-word} of tableau $T$ is defined by reading the entries of $T$ ``from left to right and bottom to top", i.e., 
starting with the bottom row, writing down its entries from left to right, then listing the entries from left to right in the next to the bottom row and working up to the top.
The row-word of $T$ is denoted by $w_{row}(T)$. 
The following two relations $\leftarrow$ and $\Leftarrow$ on $\N'$ were introduced by Worley:
\begin{itemize}
\item[$\bullet$]
$a \leftarrow b$ iff ($a < b$ or $a = b = |a|$)
\item[$\bullet$]
$a \Leftarrow b$ iff ($a < b$ or $a = b \ne |a|$)
\end{itemize}

\begin{dfn}[Knuth relation cf. {\cite[Definition6.2.1]{W}}]
Two words $w^{(0)}$ and $w^{(1)}$ over $\N'$ are related by {\it Knuth relation} if they satisfy one of the following conditions: 

\begin{enumerate}
\item[(i)]
The words $w^{(0)}$ and $w^{(1)}$ are of the form $w^{(0)} = u \cdot cab \cdot v$ and $w^{(1)} = u \cdot acb \cdot v$ (or vice-versa) with $a \leftarrow b \Leftarrow c$ and some words $u, v$ over $\N'$ (but $u, v$ can be the empty word).

\item[(ii)]
The words $w^{(0)}$ and $w^{(1)}$ are of the form $w^{(0)} = u \cdot bac \cdot v$ and $w^{(1)} = u \cdot bca \cdot v$ (or vice-versa) with $a \Leftarrow b \leftarrow c$ and some words $u, v$ over $\N'$ (but $u, v$ can be the empty word). 

\item[(iii)]
The words $w^{(0)}$ and $w^{(1)}$ are of the form $w^{(0)} = ab \cdot w$ and $w^{(1)} = ba \cdot w$ (or vice-versa) with $|a| \ne |b|$ and some word $w$ over $\N'$ (but $w$ can be the empty word). 

\item[(iv)]
The words $w^{(0)}$ and $w^{(1)}$ are of the form $w^{(0)} = a|a| \cdot w$ and $w^{(1)} = a|a|' \cdot w$ (or vice-versa) with $a \in \N'$ and some word $w$ over $\N'$ (but $w$ can be the empty word). 

\end{enumerate}

\end{dfn}

\begin{ex}
\begin{itemize}
\item
The words $12 \cdot 322 \cdot 41$ and $12 \cdot 232 \cdot 41$ are related by (i) because $2 \leftarrow 2 \Leftarrow 3$.
\item
The words $3'23'$ and $3'3'2$ are related by (i) because $2 \leftarrow 3' \Leftarrow 3'$.
\item
The words $32'11$ and $2'311$ are related by (iii) because $|3| \ne |2'|$.
\item
The words $22'11$ and $2211$ are related by (iv).
\end{itemize}
\end{ex}
For two words $w^{(0)}, w^{(1)}$, we say that $w^{(0)}$ is {\it Knuth equivalent} to $w^{(1)}$ if they are related by the transitive closure of the Knuth relations and we write $w^{(0)} \equiv w^{(1)}$. 
The following theorem is known about Knuth equivalence.

\begin{thms}[cf. {\cite[Theorem 6.2.2]{W}}]\label{inv}
Every word over $\N'$ is Knuth equivalent to the row-word of a unique normal tableau of $\QSST$. 
\end{thms}

We denote such a unique normal shifted tableau by $P(w)$, i.e., $w \equiv w_{row}(P(w))$ for a word $w$ over $\N'$.

\subsubsection{Jeu de Taquin}
Here, we review the operation {\it jeu de taquin}. It is an operation which transforms a skew shifted tableau into a normal tableau. 

First, we define the operation called {\it elementary slide}. An elementary slide is performed on a tableau $L$ over $\N'$ such that an entry is not in $\N'$ but is a dot. Thus, if $a$ is the element below the dot, and $b$ is the element to the right of the dot, then we replace the dot with $a$ if $a \leftarrow b$ and with $b$ if $b \Leftarrow a$. If either $a$ or $b$ is missing from the tableau, we replace the dot with the present entry.  Elementary slides are performed as shown in Figure \ref{slide_fig}. However, there is an exceptional case, which is shown in Figure \ref{slide_fig2}. If the dot is in the box $(i, i)$, the entry in the box $(i, i + 1)$ is $k'$, and that in $(i +1, i + 1)$ is $k$ (resp. $k'$), we put the dot into $(i + 1, i + 1)$, put $k$ into $(i, i + 1)$, and put $k$ (resp. $k'$) into $(i, i)$.

A skew shifted diagram $\lambda/\mu$ has one or more {\it inside corners}. An inside corner is a corner of the smaller (deleted) diagram $\mu$. An {\it outside corner} is a corner of $\lambda$. 
The operation of a {\it slide} takes a tableau $L = (l_{ij}) \in \QSST(\lambda/\mu)$ and an inside corner $c$ of $\lambda/\mu$. Assume that a dot is put into the inside corner $c$ of $L$. 
The slide is completed by repeating elementary slides until the dot is placed in an outside corner.
If the dot is placed in an outside corner of the tableau, we remove both the box containing the dot and the dot itself.

\begin{figure}[h]
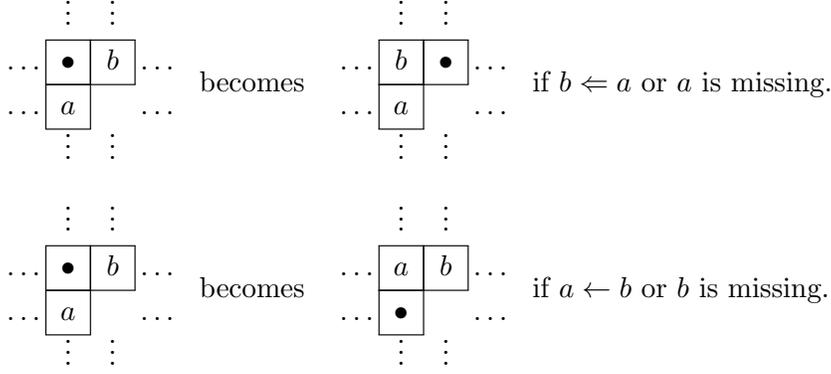
\label{slide}
\begin{itemize}[label = {}]
\item
$
\vcenter{\hbox{
~\begin{ytableau}
\none & \none[\vdots] & \none[\vdots] \\
\none[\dots] & \bullet & b & \none[\dots] \\
\none[\dots] & a & \none & \none[\dots]\\
\none & \none[\vdots] & \none[\vdots]
\end{ytableau}~}}$
becomes 
$
\vcenter{\hbox{
~\begin{ytableau}
\none & \none[\vdots] & \none[\vdots] \\
\none[\dots] & b & \bullet & \none[\dots] \\
\none[\dots] & a & \none & \none[\dots]\\
\none & \none[\vdots] & \none[\vdots]
\end{ytableau}~}}$
if $b \Leftarrow a$ or $a$ is missing.
\item
$
\vcenter{\hbox{
~\begin{ytableau}
\none & \none[\vdots] & \none[\vdots] \\
\none[\dots] & \bullet & b & \none[\dots] \\
\none[\dots] & a & \none & \none[\dots]\\
\none & \none[\vdots] & \none[\vdots]
\end{ytableau}~}}$
becomes 
$
\vcenter{\hbox{
~\begin{ytableau}
\none & \none[\vdots] & \none[\vdots] \\
\none[\dots] & a & b & \none[\dots] \\
\none[\dots] & \bullet & \none & \none[\dots]\\
\none & \none[\vdots] & \none[\vdots]
\end{ytableau}~}}$
if $a \leftarrow b$ or $b$ is missing.
\end{itemize}
\caption{Elementary slide}
\label{slide_fig}
\end{figure}

\begin{figure}[h]
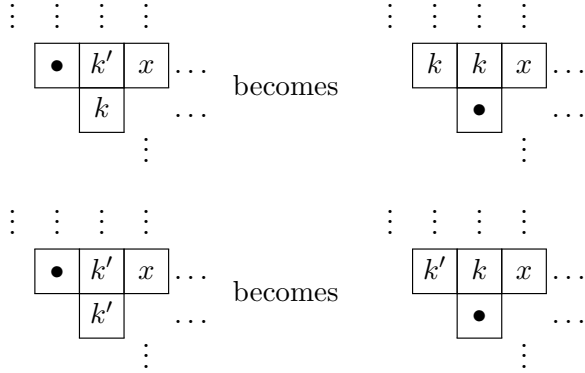
\label{slide2}
\begin{itemize}[label={}]
\item
$
\vcenter{\hbox{
~\begin{ytableau}
\none[\vdots] & \none[\vdots] & \none[\vdots] & \none[\vdots] \\
\none & \bullet & k' & x & \none[\dots] \\
\none & \none & k & \none & \none[\dots] \\
\none & \none & \none & \none[\vdots]
\end{ytableau}~}}$
becomes
$
\vcenter{\hbox{
~\begin{ytableau}
\none[\vdots] & \none[\vdots] & \none[\vdots] & \none[\vdots] \\
\none & k & k & x & \none[\dots] \\
\none & \none & \bullet & \none & \none[\dots] \\
\none & \none & \none & \none[\vdots]
\end{ytableau}~}}$
\item
$
\vcenter{\hbox{
~\begin{ytableau}
\none[\vdots] & \none[\vdots] & \none[\vdots] & \none[\vdots] \\
\none & \bullet & k' & x & \none[\dots] \\
\none & \none & k' & \none & \none[\dots] \\
\none & \none & \none & \none[\vdots]
\end{ytableau}~}}$
becomes
$
\vcenter{\hbox{
~\begin{ytableau}
\none[\vdots] & \none[\vdots] & \none[\vdots] & \none[\vdots] \\
\none & k' & k & x & \none[\dots] \\
\none & \none & \bullet & \none & \none[\dots] \\
\none & \none & \none & \none[\vdots]
\end{ytableau}~}}$
\end{itemize}
\caption{The exceptional elementary slides}
\label{slide_fig2}
\end{figure}

\begin{figure}[h]
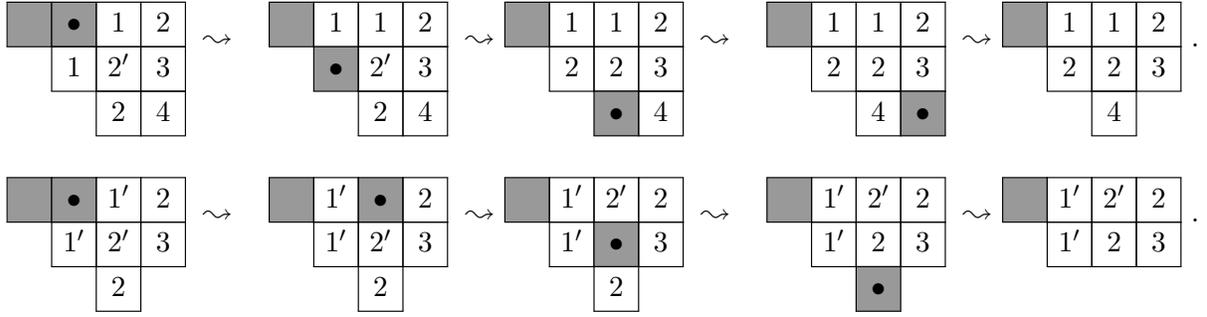

\begin{align*}
&~\begin{ytableau}
*(black!40){} & *(black!40){\bullet} & 1 & 2 \\
\none & 1 & 2' & 3 \\
\none & \none & 2 & 4
\end{ytableau}~\leadsto
&~\begin{ytableau}
*(black!40){} & 1 & 1 & 2 \\
\none & *(black!40){\bullet} & 2' & 3 \\
\none & \none & 2 & 4
\end{ytableau}~\leadsto
&~\begin{ytableau}
*(black!40){} & 1 & 1 & 2 \\
\none & 2 & 2 & 3 \\
\none & \none & *(black!40){\bullet} & 4
\end{ytableau}~\leadsto
&~\begin{ytableau}
*(black!40){} & 1 & 1 & 2 \\
\none & 2 & 2 & 3 \\
\none & \none & 4 & *(black!40){\bullet}
\end{ytableau}~\leadsto
&~\begin{ytableau}
*(black!40){} & 1 & 1 & 2 \\
\none & 2 & 2 & 3 \\
\none & \none & 4 
\end{ytableau}~. \\[4mm]
&~\begin{ytableau}
*(black!40){} & *(black!40){\bullet} & 1' & 2 \\
\none & 1' & 2' & 3 \\
\none & \none & 2 
\end{ytableau}~\leadsto
&~\begin{ytableau}
*(black!40){} & 1' & *(black!40){\bullet} & 2 \\
\none & 1' & 2' & 3 \\
\none & \none & 2 
\end{ytableau}~\leadsto
&~\begin{ytableau}
*(black!40){} & 1' & 2' & 2 \\
\none & 1' & *(black!40){\bullet} & 3 \\
\none & \none & 2 
\end{ytableau}~\leadsto
&~\begin{ytableau}
*(black!40){} & 1' & 2' & 2 \\
\none & 1' & 2 & 3 \\
\none & \none & *(black!40){\bullet} 
\end{ytableau}~\leadsto
&~\begin{ytableau}
*(black!40){} & 1' & 2' & 2 \\
\none & 1' & 2 & 3 
\end{ytableau}~.
\end{align*}
\caption{Examples of slide}
\end{figure}

Jeu de taquin is performed as a succession of {\it slides} on a tableau as follows. 
Given a tableau $L \in \QSST(\lambda/\mu)$, the process of slide can be carried out from any inside corner. When we complete the slide, another inside corner can be chosen for the resulting tableau, and the procedure is repeated, until there are no more inside corners. When we complete the whole process, the resulting tableau has a normal shifted shape. 
The whole process is called the {\it jeu de taquin}. 
The following theorem is known.

\begin{thms}[cf. {\cite[Theorem 6.4.3]{W}}]\label{rect_word}
Suppose a slide can be performed on $L \in \QSST$ to produce the tableau $M$. Then the row-word of $L$ is equivalent under the Knuth relations to the row word of $M$, i.e., 
$$w_{row}(L) \equiv w_{row}(M).$$
\end{thms}

Combining Theorem \ref{rect_word} and Theorem \ref{inv}, we have that $L \in \QSST$  becomes the normal shifted tableau $P(w_{row}(L))$ by jeu de taquin. We define $\Rect(L) = P(w_{row}(L))$ for $L \in \QSST$. The notation $\Rect$ is derived from the term rectification.

The following theorem about jeu de taquin is one of the combinatorial rules for computing the coefficients $f_{\mu\nu}^\lambda$.

\begin{thms}[cf. {\cite[Theorem 7.2.2]{W}}]\label{LRrule2}
For any $M \in \QSST(\nu)$, 
$$
\#\{L \in \QSST(\lambda/\mu) \ |\ \Rect(L) = M\} = f_{\mu\nu}^\lambda.
$$
\end{thms}

By the definition of slide, if $L$ is a skew shifted tableau of shifted shape $\lambda/\nu$, then the shape of $\Rect(L)$ is a shifted shape $\nu$ such that $SD(\nu) \subset SD(\lambda)$.

\subsection{Preparation for the proof of Theorem \ref{skew}}\label{prep_subs}

Let $\tilde{\N}=\{k_l\ |\ k, l \in \N\}$ be a set with the total order $1_1 < 1_2 < \dots < 2_1 < 2_2 <\dots $ and let $\tilde{\N}'=\{k_l, k'_l\ |\ k, l \in \N\}$ be a set with the total order
$1'_1 < 1_1 < 1'_ 2 < 1_2 < \dots < 2'_1 < 2_1 < 2'_2 < \dots .$Let $|k_l| = |k'_l| = k_l$ for $k, l \in \N$. 
Furthermore, we define the relations $\leftarrow, \Leftarrow$ on $\tilde{\N}'$ similarly to the case of $\N'$ as follows:
for $a, b \in \tilde{\N}',$
\begin{itemize}
\item[$\bullet$]
$a \leftarrow b$ iff ($a < b$ or $a = b = |a|$)
\item[$\bullet$]
$a \Leftarrow b$ iff ($a < b$ or $a = b \ne |a|$)
\end{itemize}

Based on this order and relations, we define $\QSST$ and Knuth equivalence on  words over $\tilde{\N}'$. We denote the set of these tableaux by $\QSST_{\tilde{\N}'}$ and the equivalence by $\equiv_{\tilde{\N}'}$. Furthermore, we define the operation of slide on $\QSST_{\tilde{\N}'}$ with this order. 
We identify the elements of $\tilde{\N}$ with the unmarked elements of $\tilde{\N}'$, and regard $\tilde{\N}$ as a subset of $\tilde{\N}'$. This allows us to treat words over $\tilde{\N}$ as words over $\tilde{\N}'$, that simply contain unmarked elements. Consequently, the Knuth equivalence $\equiv_{\tilde{\N}'}$ can be applied to these words. 
Analogous to Theorem \ref{inv} and Theorem \ref{rect_word} for $\QSST$, we have the following theorems for $\QSST_{\tilde{\N}'}$.

\begin{thms}[$\tilde{\N}'$ analogues of Theorem \ref{inv}]\label{inv_tilde}
Every word $\tilde{w}$ over $\tilde{\N}'$ is Knuth equivalent to the row-word of a unique normal tableau of $\QSST_{\tilde{\N}'}$. We denote such a unique normal tableau by $P(\tilde{w})$.
\end{thms}

\begin{thms}[$\tilde{\N}'$ analogues of Theorem \ref{rect_word}]\label{rect_word_tilde}
Suppose a slide can be performed on $\tilde{L} \in \QSST_{\tilde{\N}'}$ to produce the tableau $\tilde{M} \in \QSST_{\tilde{\N}'}$. Then the row-word of $\tilde{L}$ is equivalent under the Knuth relations to the row word of $\tilde{M}$, i.e., 
$$w_{row}(\tilde{L}) \equiv w_{row}(\tilde{M}).$$
\end{thms}
We can prove Theorems \ref{inv_tilde} and \ref{rect_word_tilde} by translating the Theorems \ref{inv} and \ref{rect_word} via some order-preserving maps from $\N'$ to $\tilde{\N}'$.
Then, a shifted tableau $\tilde{L} \in \QSST_{\tilde{\N}'}$ becomes the normal shifted tableau $P(w_{row}(\tilde{L}))$ by jeu de taquin, and we define $\Rect(\tilde{L}) = P(w_{row}(\tilde{L}))$.

We define the map $\phi_T$ from $\QSST$ to $\QSST_{\tilde{\N}'}$. For any $L \in \QSST$, the shifted tableau $\phi_T(L)$ is constructed as follows: let $\bm{n}=(n_1, \dots, n_r)$ be the content of $L$. (Note that $n_i \in \N$, not in $\tilde{\N}$). For each $k \in \N$, we label the entries $k'$ and $k$ in $L$ with subscript $1, 2, \dots, n_k$. These labels are assigned in increasing order, first to the  entries $k'$ top to bottom, and then to the entries $k$ from left to right. 
It is clear that for any $L \in \QSST$, we can recover $L$ by applying the labeling-off operation to $\phi_T(L)$. Thus, the restriction of $\phi_T$ to its image is invertible, and its inverse, denoted by $\phi_T^{-1}$, is the labeling-off map.

We define the map $\phi_w$ from the set of all words over $\N'$ to the set of all words over $\tilde{\N}$. For any word $w$ over $\N'$, the word $\phi_w(w)$ is constructed as follows:  let $\bm{n}=(n_1, \dots, n_r)$ be the content of $w$. For each $k\in \N$, we label the occurrences of $k'$ and $k$ in $w$ with subscript $1, 2, \dots, n_k$. These labels are assigned in increasing order, first to the  entries $k'$ right to left, and then to the entries $k$ from left to right.  Finally, remove the mark from the elements originating from $k'$.
For any $L \in \QSST$, we immediately have that $\phi_w (w_{row}(L)) = |w_{row}(\phi_T(L))|$ by the definitions.\\[-0.5em]

For example, for $
L=
~\begin{ytableau}
1' & 1 & 2' &2 \\
\none & 2' & 2 & 3 \\
\none & \none & 3
\end{ytableau}~$ with $w_{row}(L)=3 2' 2 3 1' 1 2' 2$, we have the following by the definition:
\begin{align*}
\phi_T(L)&=
~\begin{ytableau}
1'_1 & 1_2 & 2'_1 & 2_4\\
\none & 2'_2 & 2_3 & 3_2 \\
\none & \none & 3_1
\end{ytableau}~, \\
w_{row}(\phi_T(L)) &= 3_1 2'_2 2_3 3_2 1'_1 1_2 2'_1 2_4,\\
\phi_w(w_{row}(L)) &= 3_1 2_2 2_3 3_2 1_1 1_2 2_1 2_4= |w_{row}(\phi_T(L))|.
\end{align*}

We have the following proposition.
\begin{prop}\label{rect_phi}
Let $L \in \QSST$ have a skew shape.
Then, we have
$$|\Rect(\phi_T(L))| = |\phi_T(\Rect(L))|.$$
\end{prop}

To prove Proposition \ref{rect_phi}, we show the following lemma.
\begin{lem}\label{lem_phi}
Let $w^{(0)}, w^{(1)}$ be words over $\N'$.
Then, 
$$w^{(0)} \equiv w^{(1)} \Longrightarrow \phi_w(w^{(0)}) \equiv \phi_w(w^{(1)}).$$
\end{lem}

\begin{proof}[Proof of Lemma \ref{lem_phi}]
It suffices to check only the case in which $w^{(0)}$ and $w^{(1)}$ are related by a single Knuth relation. 
\\

\noindent
\textbf{\underline{Case of (i).}} Suppose $w^{(0)}$ and $w^{(1)}$ are related by (i). Then we can assume that they are of the form $w^{(0)} = u \cdot cab \cdot  v$ and $w^{(1)} = u \cdot acb \cdot v$ for some words $u, v$ over $\N'$ and $a, b, c\in \N'$ with $a \leftarrow b \Leftarrow c$. By this condition, we have $a \ne c$ because $a \le b < c$ if $b$ is unmarked and $a < b \le c$ if $b$ is marked. Let $\alpha^{(0)}, \beta^{(0)}, \gamma^{(0)} \in \N$ (resp. $\alpha^{(1)}, \beta^{(1)}, \gamma^{(1)} \in \N$) be the subscripts labeled with the entries $a, b, c$ respectively in $\phi_w(w^{(0)})$ (resp. $\phi_w(w^{(1)})$). 
Let $\tilde{u}(w^{(0)}), \tilde{v}(w^{(0)})$ (resp. $\tilde{u}(w^{(1)}), \tilde{v}(w^{(1)})$) be the words over $\tilde{\N}$ originated from $u, v$ in $\phi_w(w^{(0)})$ (resp. $\phi_w(w^{(1)})$). 
It means that $\phi_w(w^{(0)}) = \tilde{u}^{(0)} \cdot |c|_{\gamma^{(0)}} |a|_{\alpha^{(0)}} |b|_{\beta^{(0)}} \cdot \tilde{v}^{(0)}$ and $\phi_w(w^{(1)}) = \tilde{u}^{(1)} \cdot |a|_{\alpha^{(1)}} |c|_{\gamma^{(1)}} |b|_{\beta^{(1)}} \cdot \tilde{v}^{(1)}$. Now, we prove that $\phi_w(w^{(1)}) = \tilde{u}^{(0)} \cdot |a|_{\alpha^{(0)}} |c|_{\gamma^{(0)}} |b|_{\beta^{(0)}} \cdot \tilde{v}^{(0)}$.

\begin{itemize}
\item
\underline{$\tilde{u}^{(0)} = \tilde{u}^{(1)}$.}
By the definition of $\phi_w$, the subscripts in $\tilde{u}^{(0)}$ (resp. $\tilde{u}^{(1)}$) are determined only by the number of each marked integers in $cab \cdot v$ (resp. $acb \cdot v$). When we count the numbers, there is no difference between the two cases: in $cab \cdot v$ and in $acb \cdot v$. So we have $\tilde{u}^{(0)} = \tilde{u}^{(1)}$.
 
\item
\underline{$\tilde{v}^{(0)} = \tilde{v}^{(1)}$.}
The subscripts labeled with $k$ in $\tilde{v}^{(0)}$ (resp. $\tilde{v}^{(1)}$) are determined only by the number of $k'$ in $w^{(0)}$ (resp. $w^{(1)}$) and the number of $k$ in $u \cdot cab$ (resp. $u \cdot acb$) for each $k \in \N$. When we count the numbers, there is no difference between the two cases: in $u \cdot cab$ and in $u \cdot acb$. So we have $\tilde{v}^{(0)} = \tilde{v}^{(1)}$.

\item
\underline{$\alpha^{(0)} = \alpha^{(1)}$.}
If $a$ is marked, $\alpha^{(0)}$ (resp. $\alpha^{(1)}$) is the number of element $a$ in $b \cdot v$ (resp. $cb \cdot v$) plus 1. By the condition $a \ne c$, we have that $\alpha^{(0)} = \alpha^{(1)}$. 
If $a$ is unmarked, $\alpha^{(0)}$ (resp. $\alpha^{(1)}$) is the sum of the number of element $a'$ in $w^{(0)}$ (resp. $w^{(1)}$), the number of element $a$ in $u \cdot c$ (resp. $u$) and $1$. Because $w^{(0)}$ and $w^{(1)}$ have the same components and $a \ne c$, we have $\alpha^{(0)} = \alpha^{(1)}$.

\item
\underline{$\beta^{(0)} = \beta^{(1)}$.}
If $b$ is marked, both $\beta^{(0)}$ and $\beta^{(1)}$ are the number of element $b$ in $v$ plus 1. So we have that $\beta^{(0)} = \beta^{(1)}$. 
If $b$ is unmarked, $\beta^{(0)}$ (resp. $\beta^{(1)}$) is the sum of the number of element $b'$ in $w^{(0)}$ (resp. $w^{(1)}$), the number of element $b$ in $u \cdot ca$ (resp. $u \cdot ac$) and $1$. Because $w^{(0)}$ and $w^{(1)}$, $u \cdot ca$ and $u \cdot ac$ have the same components respectively, so we have $\beta^{(0)} = \beta^{(1)}$.

\item
\underline{$\gamma^{(0)} = \gamma^{(1)}$.}
If $c$ is marked, $\gamma^{(0)}$ (resp. $\gamma^{(1)}$) is the number of element $c$ in $ab \cdot v$ (resp. $b \cdot v$) plus 1. By the condition $a \ne c$, we have that $\gamma^{(0)} = \gamma^{(1)}$. 
If $c$ is unmarked, $\gamma^{(0)}$ (resp. $\gamma^{(1)}$) is the sum of the number of element $c'$ in $w^{(0)}$ (resp. $w^{(1)}$), the number of element $c$ in $u$ (resp. $u \cdot a$) and $1$. Because $w^{(0)}$ and $w^{(1)}$ have the same components and $a \ne c$, we have $\gamma^{(0)} = \gamma^{(1)}$.

\end{itemize}
Now we have $\phi_w(w^{(0)}) = \tilde{u}^{(0)} \cdot |a|_{\alpha^{(0)}} |c|_{\gamma^{(0)}} |b|_{\beta^{(0)}} \cdot \tilde{v}^{(0)}$. 
If $a_{\alpha^{(0)}} < b_{\beta^{(0)}} < c_{\gamma^{(0)}}$, then we have that $\phi_w(w^{(1)})$ is related to $\phi_w(w^{(0)})$ by (i).
Now we show $a_{\alpha^{(0)}} < b_{\beta^{(0)}} < c_{\gamma^{(0)}}$.
If $a < b < c$, we have $a_{\alpha^{(0)}} < b_{\beta^{(0)}} < c_{\gamma^{(0)}}$ immediately. 
If $a = b < c$, then $a, b$ are unmarked because of the condition $a \leftarrow b$. So $\alpha^{(0)} < \beta^{(0)}$ because $a$ is left to $b$ in $w^{(0)}$. Then we have $a_{\alpha^{(0)}} < b_{\beta^{(0)}} < c_{\gamma^{(0)}}$.
If $a < b = c$, then $b, c$ are marked because of the condition $b \Leftarrow c$. So $\gamma^{(0)} < \beta^{(0)}$ because $b$ is right to $c$ in $w^{(0)}$. Then we have $a_{\alpha^{(0)}} < b_{\beta^{(0)}} < c_{\gamma^{(0)}}$.
Thus, we obtain the result of the case (i).
\\

\noindent
\textbf{\underline{Case of (ii).}} Suppose $w^{(0)}$ and $w^{(1)}$ are related by (ii). Then we can assume that they are of the form $w^{(0)} = u \cdot bac \cdot  v$ and $w^{(1)} = u \cdot bca \cdot v$ for some words $u, v$ over $\N'$ and $a, b, c\in \N'$ with $a \Leftarrow b \leftarrow c$. By this condition, we have $a \ne c$ because $a < b \le c$ if $b$ is unmarked and $a \le b < c$ if $b$ is marked. Let $\alpha^{(0)}, \beta^{(0)}, \gamma^{(0)} \in \N$ (resp. $\alpha^{(1)}, \beta^{(1)}, \gamma^{(1)} \in \N$) be the subscripts labeled with the entries $a, b, c$ respectively in $\phi_w(w^{(0)})$ (resp. $\phi_w(w^{(1)})$). 
Let $\tilde{u}^{(0)}, \tilde{v}^{(0)}$ (resp. $\tilde{u}^{(1)}, \tilde{v}^{(1)}$) be the words over $\tilde{\N}$ originated from $u, v$ in $\phi_w(w^{0})$ (resp. $\phi_w(w^{1})$). 
It means that $\phi_w(w^{(0)}) = \tilde{u}(w^{(0)}) \cdot |b|_{\beta^{(0)}} |a|_{\alpha^{(0)}} |c|_{\gamma^{(0)}} \cdot \tilde{v}^{(0)}$ and $\phi_w(w^{(1)}) = \tilde{u}(w^{(1)}) \cdot |b|_{\beta^{(1)}} |c|_{\gamma^{(1)}} |a|_{\alpha^{(1)}} \cdot \tilde{v}^{(1)}$. We can obtain that $\phi_w(w^{(1)}) = \tilde{u}^{(0)} |b|_{\beta^{(0)}} \cdot |c|_{\gamma^{(0)}} |a|_{\alpha^{(0)}} \cdot \tilde{v}^{(0)}$ and $a_{\alpha^{(0)}} < b_{\beta^{(0)}} < c_{\gamma^{(0)}}$ similarly to the case of (i). 
Then we have the conclusion that  $\phi_w(w^{(1)})$ is related to $\phi_w(w^{(0)})$ by (ii). 
\\

\noindent
\textbf{\underline{Case of (iii).}} Suppose $w^{(0)}$ and $w'$ are related by (iii). Then we can assume that they are of the form $w = ab \cdot  v$ and $w' = ba \cdot v$ for some words $v$ over $\N'$ and $a, b\in \N'$ with $|a| \ne |b|$. Let $\alpha^{(0)}, \beta^{(1)} \in \N$ (resp. $\alpha^{(1)}, \beta^{(1)} \in \N$) be the subscripts labeled with the entries $a, b$ respectively in $\phi_w(w^{(0)})$ (resp. $\phi_w(w^{(1)})$). 
Let $\tilde{v}^{(0)}$ (resp. $\tilde{v}^{(1)}$) be the words over $\tilde{\N}$ originated from $v$ in $\phi_w(w^{(0)})$ (resp. $\phi_w(w^{(1)})$). 
It means that $\phi_w(w^{(0)}) = |a|_{\alpha^{(0)}} |b|_{\beta^{(0)}} \cdot \tilde{v}^{(0)}$ and $\phi_w(w^{(1)}) = |b|_{\beta^{(1)}} |a|_{\alpha^{(1)}} \cdot \tilde{v}^{(1)}$. Now, we prove that $\phi_w(w^{(1)}) = |b|_{\beta^{(0)}} |a|_{\alpha^{(0)}} \cdot \tilde{v}^{(0)}$.

\begin{itemize}
\item
\underline{$\tilde{v}^{(0)} = \tilde{v}^{(1)}$.}
The subscripts labeled with $k$ in $\tilde{v}^{(0)}$ (resp. $\tilde{v}^{(1)}$) are determined only by the number of $k'$ in $w^{(0)}$ (resp. $w^{(1)}$) and the number of $k$ in $ab$ (resp. $ba$) for each $k \in \N$. Because $w^{(0)}$ and $w^{(1)}$, $ab$ and $ba$ have the same components respectively, so we have $\tilde{v}^{(0)} = \tilde{v}^{(1)}$.
 
\item
\underline{$\alpha^{(0)} = \alpha^{(1)}$.}
By the condition $|a| \ne |b|$, both $\alpha^{(0)}$ and $\alpha^{(1)}$ are the number of element $a$ in $v$ plus 1.  So we have that $\alpha^{(0)} = \alpha^{(1)}$. 

\item
\underline{$\beta^{(0)} = \beta^{(1)}$.}
By the condition $|a| \ne |b|$, both $\beta^{(0)}$ and $\beta^{(1)}$ are the number of element $b$ in $v$ plus 1. So we have $\beta^{(0)} = \beta^{(1)}$.

\end{itemize}
Then, we have that $\phi_w(w^{(1)}) = |b|_{\beta^{(0)}} |a|_{\alpha^{(0)}} \cdot \tilde{v}^{(0)}$.
Here, $|a|_{\alpha^{(0)}} \ne |b|_{\beta^{(0)}}$ because of $|a| \ne |b|$, we obtain the result of the case (iii).
\\

\noindent
\textbf{\underline{Case of (iv).}} 
Suppose $w^{(0)}$ and $w^{(1)}$ are related by (iv).  Then we can assume that they are of the form $w^{(0)} = aa \cdot u$ (resp. $a'a \cdot u$) and $w^{(1)} = aa' \cdot u$ (resp. $a'a' \cdot u$) for some words $u$ over $\N'$ and $a \in \N$. Let $n$ be the number of the entries $a'$ in $u$. Then, $\phi_w(w^{(0)})$ is of the form $\phi_w(w^{(0)}) = a_{n+1} a_{n+2} \cdot \tilde{u}$ and $\phi_w(w^{(1)})$ is of the form $\phi_w(w^{(1)}) = a_{n+2} a_{n+1} \cdot \tilde{u}$ by the definition of $\phi_w$ where $\tilde{u}$ originates from $u$. Therefore, by $|a_{n+1}| < |a_{n+2}|$, $\phi_w(w^{(0)})$ is related to $\phi_w(w^{(1)})$ by (iii). 
\end{proof}

Then, we prove Theorem \ref{rect_phi} by using Lemma \ref{lem_phi}.

\begin{proof}[Proof of Proposition \ref{rect_phi}]
For a word $w^{(0)}$ such that the set of all letters in $|w^{(0)}|$ is pairwise distinct, 
we have $|w^{(0)}| \equiv_{\tilde{\N}'} |w^{(1)}|$ for any word $w^{(1)}$ such that $w^{(0)} \equiv_{\tilde{\N}'} w^{(1)}$. This equivalence $|w^{(0)}| \equiv_{\tilde{\N}'} |w^{(1)}|$ is obtained by verifying the four types of Knuth relations (i)-(iv), respectively as follows:

\begin{enumerate}
\item[(i)]
The words $w^{(0)}$ and $w^{(1)}$ are of the form $w^{(0)} = u \cdot cab \cdot v$ and $w^{(1)} = u \cdot acb \cdot v$ (or vice-versa) with $a \leftarrow b \Leftarrow c$ and some words $u, v$ over $\N'$

$\Rightarrow$ 
The words $|w^{(0)}|$ and $|w^{(1)}|$ are of the form $|w^{(0)}| = |u| \cdot |c||a||b| \cdot |v|$ and $|w^{(1)}| = |u| \cdot |a||c||b| \cdot |v|$ (or vice-versa) with $|a| < |b| < |c|$ (because $\{|a|, |b|, |c|\}$ is pairwise distinct). 

\item[(ii)]
The words $w^{(0)}$ and $w^{(1)}$ are of the form $w^{(0)} = u \cdot bac \cdot v$ and $w^{(1)} = u \cdot bca \cdot v$ (or vice-versa) with $a \Leftarrow b \leftarrow c$ and some words $u, v$ over $\N'$

$\Rightarrow$
The words $|w^{(0)}|$ and $|w^{(1)}|$ are of the form $|w^{(0)}| = |u| \cdot |b||a||c| \cdot |v|$ and $|w^{(1)}| = |u| \cdot |b||c||a| \cdot |v|$ with $|a| < |b| < |c|$. 

\item[(iii)]
The words $w^{(0)}$ and $w^{(1)}$ are of the form $w^{(0)} = ab \cdot w$ and $w^{(1)} = ba \cdot w$ (or vice-versa) with $|a| \ne |b|$ and some word $w$ over $\N'$

$\Rightarrow$
The words $|w^{(0)}|$ and $|w^{(1)}|$ are of the form $|w^{(0)}| = |a||b| \cdot |w|$ and $|w^{(1)}| = |b||a| \cdot |w|$ (or vice-versa) with $||a|| = |a| \ne ||b|| = |b|$. 

\item[(iv)]
The words $w^{(0)}$ and $w^{(1)}$ are not related by (iv), because all letters in $|w^{(0)}|$ is pairwise distinct and then $w^{(0)}$ never be of the form $w^{(0)} = a|a| \cdot w$ or $w^{(0)} = a|a|' \cdot w$ for any $a \in \N'$ and a word $w$ over $\N'$.
\end{enumerate}
By the definition of $\phi_T$, the set of all letters in $w_{row}(\phi_T(L))$ is pairwise distinct. 
Then we have that
$$|w_{row}(\Rect(\phi_T(L)))| \equiv_{\tilde{\N}'} |w_{row}(\phi_T(L))|.$$
by
$$w_{row}(\Rect(\phi_T(L))) \equiv_{\tilde{\N}'} w_{row}(\phi_T(L)), $$
obtained by Theorem \ref{rect_word_tilde}.
Therefore, we have that 
\begin{align*}
w_{row}(|\Rect(\phi_T(L))|)
& = |w_{row}(\Rect(\phi_T(L)))|\\
& \equiv_{\tilde{\N}'} |w_{row}(\phi_T(L))|\\
& = w_{row}(|\phi_T(L)|). 
\end{align*}
By the definition of $\phi_T$ and $\phi_w$, we have 
$$\phi_w(w_{row}(L)) = w_{row}(|\phi_T(L)|).$$
Also, combining Theorem \ref{rect_word} and Lemma \ref{lem_phi}, we have
$$\phi_w(w_{row}(L)) \equiv_{\tilde{\N}'} \phi_w(w_{row}(\Rect(L))).$$
Then we have
\begin{align*}
w_{row}(|\Rect(\phi_T(L))|)
& \equiv_{\tilde{\N}'} w_{row}(|\phi_T(L)|)\\
& = \phi_w(w_{row}(L))\\
& \equiv_{\tilde{\N}'} \phi_w(w_{row}(\Rect(L)))\\
& = w_{row}(|\phi_T(\Rect(L))|).
\end{align*}
By Theorem \ref{inv_tilde}, we have
$|\Rect(\phi_T(L))| = |\phi_T(\Rect(L))|.$
\end{proof}

\begin{dfn}[Shifted arm and shifted body]\label{arm_body_def}
Let $\lambda/\mu$ be a skew strict partition. 
\begin{itemize}
\item The {\it shifted arm} of diagram $\lambda/\mu$ is the part of the diagram which consists of the boxes $(1, m + \mu_1 + 1), \dots, (1, \lambda_1)$ with $m = {\rm min}\{\lambda_2, \sum_{i \ge 2}(\lambda_i - \mu_i)\}$. We denote the shifted arm of $\lambda/\mu$ by $SA(\lambda/\mu)$. 
\item
The {\it shifted body} of $\lambda/\mu$ is the part of the diagram obtained by removing $SA(\lambda/\mu)$ from the diagram $\lambda/\mu$, and is denoted by $SB(\lambda/\mu)$.
\end{itemize}
Here, $\sum_{i \ge 2}(\lambda_i - \mu_i)$ is the number of boxes in the second row or below in the shifted diagram of $\lambda/\mu$.
\end{dfn}

\begin{ex}
For $\lambda = (6, 3)$ and $\mu = (2, 1)$, $m = {\rm min}\{\lambda_2, \sum_{i \ge 2}(\lambda_i - \mu_i)\} = {\rm min}\{3, 2\} = 2$ and then $SA(\lambda/\mu)$ consists of the boxes $\{(1, m + \mu_1 + 1), \dots, (1, \lambda_1)\} = \{(1, 5), (1, 6)\}$. The part of light gray boxes in the following figure is $SB(\lambda/\mu)$, and the part of white boxes is $SA(\lambda/\mu)$.
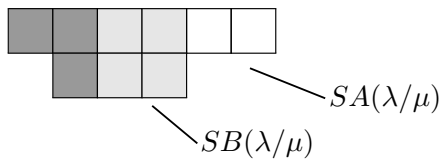
$$
\vcenter{\hbox{
\begin{tikzpicture}[baseline=(current bounding box.center)]
\node at (0,0){ 
~\begin{ytableau}
*(black!40){} & *(black!40){} & *(black!10){} & *(black!10){} & {} & {}\\
\none & *(black!40){} & *(black!10){} & *(black!10)
\end{ytableau}~};
\draw[line width=0.7pt] (1.5, -0.2) -- (2.5, -0.6) node[at={(3.3, -0.6)}]{$SA(\lambda/\mu)$};
\draw[line width=0.7pt] (0.2, -0.7) -- (0.8, -1.2) node[at={(1.6, -1.2)}]{$SB(\lambda/\mu)$};
\end{tikzpicture}}}
$$
\end{ex}

For any set $X$ and any tableau $T \in ST(\lambda/\mu, X)$, we define $SA(T)$ as the part of $T$ consisting of the boxes in $SA(\lambda/\mu)$ and entries in them. We call this the {\it shifted arm} of $T$. We also define $SB(T)$ similarly. Note that if $m + \mu_1 + 1 > \lambda_1$, $SA(\lambda/\mu)$ is empty.

\begin{lem}\label{skew_box_lem}
Let $L \in \QSST(\lambda/\mu)$, $\tilde{L}=(\tilde{l}_{ij}) \coloneq \phi_T(L)$, and $\nu = \shape(\Rect(\tilde{L}))$. Then, 
$\Rect(\tilde{L})$ has $SA(\tilde{L})$ in the right-most contiguous boxes of the first row. That is, with $m = {\rm min}\{\lambda_2, \sum_{i \ge 2}(\lambda_i - \mu_i)\}$, the entries $\tilde{l}_{1, m + \mu_1 + 1}, \dots, \tilde{l}_{1, \lambda_1}$ are placed in the boxes $(1, m + \mu_1 - \lambda_1 + \nu_1 + 1), \dots, (1, \nu_1)$ in $\Rect(\tilde{L})$, respectively.
\end{lem}

\begin{ex}
$$
\vcenter{\hbox{
\begin{tikzpicture}[baseline=(current bounding box.center)]
\node at (0,0){ 
~\begin{ytableau}
*(black!40){} & *(black!40){} & *(black!40){} & *(black!10)1_1 & *(black!10)2'_1 & *(black!10)3_1 & 4_3 & 4_4 \\
\none & *(black!40){} & *(black!10)2'_2 & *(black!10)4_2  \\
\none & \none & *(black!10)4_1
\end{ytableau}~};
\draw[line width=0.7pt] (1.7, 0.2) -- (2.2, -0.4) node[at={(2.8, -0.4)}]{$SA(\tilde{L})$};
\draw[line width=0.7pt] (0.2, -0.2) -- (0.8, -0.8) node[at={(1.4, -0.8)}]{$SB(\tilde{L})$};
\end{tikzpicture} }}
\stackrel{\text{jeu de taquin}}{\mapsto}
\vcenter{\hbox{
\begin{tikzpicture}[baseline=(current bounding box.center)]
\node at (0,0){ 
~\begin{ytableau}
*(black!10)1_1 & *(black!10)2'_1 & *(black!10)3_1 & *(black!10)4_2 & 4_3 & 4_4 \\
\none & *(black!10)2'_2 & *(black!10)4_1  
\end{ytableau}~};
\draw[line width=0.7pt] (1, -0.1) -- (1.3, -0.5) node[at={(1.9, -0.5)}]{$SA(\tilde{L})$};
\end{tikzpicture}. }}$$

\end{ex}

\begin{proof}[Proof of Lemma \ref{skew_box_lem}]
Consider the process of jeu de taquin which is carried out starting from the slide of the right-most entry of the bottom row of the smaller (deleted) diagram in each step.
$$
~
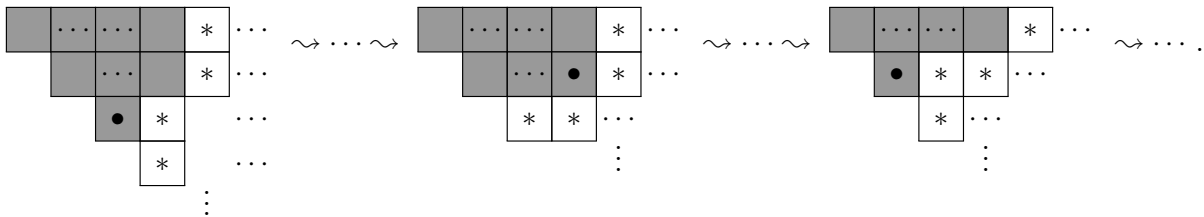
\begin{ytableau}
*(black!40){} & *(black!40)\cdots & *(black!40)\cdots & *(black!40){} & {\ast} & \none[\cdots]\\
\none & *(black!40){} & *(black!40)\cdots & *(black!40){} & {\ast} & \none[\cdots]\\
\none & \none & *(black!40)\bullet & {\ast} & \none & \none[\cdots] \\
\none & \none & \none & {\ast} & \none & \none[\cdots] \\
\none & \none & \none & \none & \none[\vdots] 
\end{ytableau}~
\leadsto
\cdots 
\leadsto
~\begin{ytableau}
*(black!40){} & *(black!40)\cdots & *(black!40)\cdots & *(black!40){} & {\ast} & \none[\cdots]\\
\none & *(black!40){} & *(black!40)\cdots & *(black!40)\bullet & {\ast} & \none[\cdots]\\
\none & \none & {\ast} & {\ast} & \none[\cdots] \\
\none & \none & \none & \none & \none[\vdots] 
\end{ytableau}~
\leadsto 
\cdots
\leadsto
~\begin{ytableau}
*(black!40){} & *(black!40)\cdots & *(black!40)\cdots & *(black!40){} & {\ast} & \none[\cdots]\\
\none & *(black!40)\bullet & {\ast} & {\ast} & \none[\cdots]\\
\none & \none & {\ast} & \none[\cdots] \\
\none & \none & \none & \none[\vdots] 
\end{ytableau}~
\leadsto
\cdots.
$$
Now, suppose that we complete the slides from the second row and below. Let $\tilde{L}'$ be the resulting tableau after these slides are completed. By the definition of slide, the first row of the tableau is not changed throughout the process. Therefore, now the shape of $\tilde{L}'$ is $\psi/(\mu_1)$ for some normal shifted shape $\psi = (\psi_1, \dots, \psi_r)$ and the shape $(\mu_1)$ which consist of one row.
$$
\tilde{L'} = 
\ytableausetup{boxsize=1.7em}
\begin{tikzpicture}[baseline=(current bounding box.center)]
\node at (0,0){ 
~\begin{ytableau}
*(black!40){} & *(black!40)\cdots & *(black!40)\cdots & *(black!40){} & \tilde{l}_{1i} & \cdots & \tilde{l}_{1 \lambda_1}\\
\none & \ast & \ast & \ast & \none[\cdots]\\
\none & \none & \ast & \none & \none[\cdots]\\
\none & \none & \none & \none[\vdots] 
\end{ytableau}~
};
\node at (1.5, 1.6) {\small ($i = \mu_1 + 1$)};
\end{tikzpicture} 
\in \QSST_{\tilde{\N}}(\psi). 
$$
Since the entire sliding process is contained within the outer shape $\lambda$, the length of any row in $\tilde{L}'$ cannot exceed the length of the corresponding row in $\lambda$. In particular, for the second row, we must have $\psi_2 \le \lambda_2$.
Furthermore, the $\psi_2$ entries in the second row of $\tilde{L}'$ must have originated from the second row or below in the original tableau $\tilde{L}$. The total number of entries in these rows is $\sum_{i \ge 2}(\lambda_i -\mu_i)$. 
Therefore, we have $\psi_2 \le \sum_{i \ge 2}(\lambda_i - \mu_i)$.
Combining $\psi_2 \le \lambda_2$ and $\psi_2 \le \sum_{i \ge 2}(\lambda_i - \mu_i)$, we obtain that $\psi_2 \le {\rm min}\{\lambda_2, \sum_{i \ge 2}(\lambda_i - \mu_i)\}$. 
Let $m = {\rm min}\{\lambda_2, \sum_{i \ge 2}(\lambda_i - \mu_i)\}$.

Now, let us consider the remaining slides of jeu de taquin on $\tilde{L}$, which is equivalent to the process of jeu de taquin on $\tilde{L}'$. In this process, the entries in the first row slide to the left of the first row or remain in their position, and their relative horizontal order is preserved. 
Since the second row of $\tilde{L}'$ has $\psi_2$ boxes on the right of one space, there are no entries in the $(i, j)$-box with $i \ge  2$ and $j \ge \psi_2 + 2$ in $\tilde{L}'$. 
Since an elementary slide moves an entry only into an adjacent box to its left or above it, no entry from the second row or below can move into the $(1, j)$-box in $\tilde{L}'$ with $j \ge \psi_2 + 2$. 
Then, we obtain that the entries $\tilde{l}_{1, \psi_2 - \nu_1 + \lambda_1 + 2}, \dots, \tilde{l}_{1\lambda_1}$ are placed in the boxes $(1, \psi_2 + 2), \dots, (1, \nu_1)$, respectively, where $\nu = \shape(\Rect(\tilde{L}))$.

It follows that $\nu_1 \ge \lambda_1 - \mu_1$, since all entries from the first row of $\tilde{L}$ move into the first row of $\Rect(\tilde{L})$ by jeu de taquin.  
We now consider two cases: $\nu_1 > \lambda_1 - \mu_1$ and $\nu_1 = \lambda_1 - \mu_1$ separately.
If $\nu_1 > \lambda_1 - \mu_1$, then we have the following calculation: 
\begin{align*}
\psi_2 - \nu_1 + \lambda_1 + 2
&\le \psi_2 - \nu_1 + \lambda_1 + (\nu_1 - \lambda_1 + \mu_1) + 1 \\
&= \psi_2 + \mu_1 + 1\\
&\le m + \mu_1 + 1
\end{align*}
by $1 \le \nu_1 - \lambda_1 + \mu_1$ and $\psi_2 \le m$. Then, we have that the entries $\tilde{l}_{1, m + \mu_1 + 1}, \dots, \tilde{l}_{1\lambda_1}$ are placed in the boxes $(1, m + \mu_1 - \lambda_1 + \nu_1 + 1), \dots, (1, \nu_1)$ in $\Rect(\tilde{L})$, respectively. 
Next, we suppose that $\nu_1 = \lambda_1 - \mu_1$. In this case, the process of jeu de taquin on $\tilde{L}'$ was sliding the entries only leftward. Therefore, the entries $\tilde{l}_{1\mu_1}, \dots, \tilde{l}_{1\lambda_1}$ are placed in the boxes $(1, 1), \dots, (1, \nu_1)$ in $\Rect(\tilde{L})$, respectively. In particular, we have that the entries $\tilde{l}_{1, m + \mu_1 + 1}, \dots, \tilde{l}_{1 \lambda_1}$ are placed in the boxes $(1, m + \mu_1 - \lambda_1 + \nu_1 + 1), \dots, (1, \nu_1)$ in $\Rect(\tilde{L})$, respectively. 
Thus, we obtain Lemma \ref{skew_box_lem}. 
\end{proof}

\subsection{Proof of Theorem \ref{skew}}\label{proof_skew_subs}

Let $\bm{v} = (v_{ij}) \in ST(\lambda/\mu, \C), L \in \SSYT(\lambda/\mu)$. Let $\nu = \shape(\Rect(L))$. We define a new tableau $\bm{v}_L \in ST(\nu, \C)$ as follows. 
Let $\tilde{L} = (\tilde{l}_{ij}) \coloneq \phi_T(L)$. The definition of $\phi_T$ ensures that the set of all entries in $|\tilde{L}|$ is pairwise distinct. By the definition of jeu de taquin,  
the set of entries of $|\tilde{L}|$ and that of $|\Rect(\tilde{L})|$ are the same.  Therefore, there is a natural bijection $\rho_L : SD(\lambda/\mu) \to SD(\nu)$, 
which maps $(i, j) \in SD(\lambda/\mu)$ to the box in $\Rect(\tilde{L})$ where $|\tilde{l}_{ij}|$ or $|\tilde{l}_{ij}|'$ move into by jeu de taquin on $\tilde{L}$. This means that $|\Rect(\tilde{L})| = (|\tilde{l}_{\rho_L^{-1}(ij)}|)_{(i, j) \in SD(\nu)}$.
Using $\rho_L$, we define $\bm{v}_{L} = (v_{\rho_L^{-1}(ij)})_{(i, j) \in SD(\nu)}$. It means that we place the entry  $v_{i_0j_0}$ into the box at $(i_1, j_1)$ in $\bm{v}_{L}$ when $(i_0, j_0)$ in $\tilde{L}$ is mapped to $(i_1, j_1)$ in $\Rect(\tilde{L})$. 
Note that $\bm{v}_L$ has the same shape as $\Rect(L)$ and its entries are complex variables from the set $\{v_{ij}\}$, with each variable appearing exactly once.

\ytableausetup{boxsize=1.6em}
\begin{ex}
For $L =
~\begin{ytableau}
*(black!40){} & *(black!40){} & *(black!40){} & 1 & 2' & 3 & 4 & 4 \\
\none & *(black!40){} & 2' & 4  \\
\none & \none & 2'
\end{ytableau}~$ and 
$\bm{v} =
~\begin{ytableau}
*(black!40){} & *(black!40){} & *(black!40){} & v_{14} & v_{15} & v_{16} & v_{17} & v_{18} \\
\none & *(black!40){} & v_{23} & v_{24}  \\
\none & \none & v_{33}
\end{ytableau}~$,

\begin{align*}
\phi_T(L) 
=
~\begin{ytableau}
*(black!40){} & *(black!40){} & *(black!40){} & 1_1 & 2'_1 & 3_1 & 4_2 & 4_3 \\
\none & *(black!40){} & 2'_2 & 4_1  \\
\none & \none & 2'_3
\end{ytableau}~ \stackrel{\text{jeu de taquin}}{\mapsto}
~\begin{ytableau}
1_1 & 2'_1 & 2_3 & 3_1 & 4_2 & 4_3 \\
\none & 2'_2 & 4_1  
\end{ytableau}~
=\Rect(\phi_T(L)).\\
\end{align*}
Then,
\begin{align*}
\bm{v} 
=
~\begin{ytableau}
*(black!40){} & *(black!40){} & *(black!40){} & v_{14} & v_{15} & v_{16} & v_{17} & v_{18} \\
\none & *(black!40){} & v_{23} & v_{24}  \\
\none & \none & v_{33}
\end{ytableau}~\qquad \mapsto \qquad
~\begin{ytableau}
v_{14} & v_{15} & v_{16} & v_{23} & v_{17} & v_{18} \\
\none & v_{22} & v_{33} 
\end{ytableau}~
=\bm{v}_L. \qquad
\end{align*}

\end{ex}
We define $\displaystyle{U_\nu^Q(\bm{v})=\{\bm{v}_L\>|\>\shape(\Rect(L))=\nu\}}.$ 
By Lemma \ref{skew_box_lem} and the definition of $\bm{v}_L$, we have the following corollary.

\begin{cor}\label{skew_box_cor}
Let $L \in \QSST(\lambda/\mu)$ and $\bm{v} = (v_{ij}) \in ST(\lambda/\mu)$. Then, 
The tableau $\bm{v}_L$ has $SA(\bm{v})$ in the right-most contiguous boxes of the first row, i.e. with $m = {\rm min}\{\lambda_2, \sum_{i \ge 2}(\lambda_i - \mu_i)\}$, the entries $v_{1, m + \mu_1 + 1}, \dots, v_{1\lambda_1}$ are placed in the boxes $(1, m + \mu_1 - \lambda_1 + \nu_1 + 1), \dots, (1, \nu_1)$ in $\bm{v}_L$, respectively.
\end{cor}

We now prove Theorem \ref{skew}.

\begin{proof}[Proof of Theorem \ref{skew}]
Let $\bm{u}_{\nu}(\bm{v})\in U_\nu(\bm{v})$ for $\nu \in\mathcal{G}^Q(\lambda/\mu)$. 
For $L \in \QSST(\lambda/\mu)$ such that $\shape(\Rect(L))$ is $\nu$, 
we have
\begin{align*}
\frac{1}{L^{\bm{v}}}
&=\frac{1}{|\phi_T^{-1}(\phi_T (L))|^{\bm{v}} }\\
&=\frac{1}{{|\phi_T^{-1}(\Rect(\phi_T (L)))|}^{\bm{v}_L} }\\
&=\frac{1}{{(\phi_T^{-1}|\Rect(\phi_T (L))|})^{\bm{v}_L} }\\
&=\frac{1}{|\Rect(L)|^{\bm{v}_{L}} }
\end{align*}
where the second and fourth equalities are obtained by definition of $\bm{v}_{L}$ and Proposition \ref{rect_phi} respectively. Also, there is only a difference between the constructions of $|\phi_T^{-1}(\Rect(\phi_T (L)))|$ and $\phi_T^{-1}|\Rect(\phi_T (L))|$ whether the subscripts are removed first or the marks are removed first from $\Rect (\phi_T (L))$. Therefore, $|\phi_T^{-1}(\Rect(\phi_T (L)))| = \phi_T^{-1}(|\Rect(\phi_T (L))|)$.

Let $m = \displaystyle{{\rm min}\{\lambda_2, \sum_{i \ge 2}(\lambda_i - \mu_i)\}}$.
From Lemma \ref{skew_box_lem} and the definition of $\bm{v}_L$, if $m + \mu_1 + 1 < \lambda_1$, the entries $|l_{1, m + \mu_1 + 1}|, \dots, |l_{1 \lambda_1}|$ are placed in the boxes $(1, m + \mu_1 - \lambda_1 + \nu_1 + 1), \dots, (1, \nu_1)$ in $|\Rect(L)|$ respectively. 

From Corollary \ref{skew_box_cor}, 
any $\bm{v}_L \in U_\nu(\bm{v})$ has $SA(\bm{v})$ in the right-most contiguous boxes in the first row. The entries in $SB(\bm{v})$ appear exactly once in the parts obtained by removing $SA(\bm{v})$ from $\bm{v}_L$. 
Therefore, $\bm{v}_L$ and $\bm{u}_\nu(\bm{v})$ can be mapped to each other by permuting variables in $SB(\bm{v})$. 
Thus, we have
$$\sum_{\Sym(SB(\bm{v}))} \frac{1}{{|\Rect(L)|}^{\bm{v}_L}}
=\sum_{\Sym(SB(\bm{v}))} \frac{1}{{|\Rect(L)|}^{{\bm{u}}_\nu(\bm{v})}}$$
for any $L \in \QSST(\lambda/\mu)$ with $\shape(\Rect(L))=\nu$.
Then, 
\begin{align*}
\sum_{\Sym(SB(\bm{v}))} \zeta_{\lambda/\mu}^Q(\bm{v})
&=\sum_{\Sym(SB(\bm{v}))} \sum_{L \in \QSST(\lambda/\mu)} \frac{1}{|L|^{\bm{v}}}\\
&=\sum_{\Sym(SB(\bm{v}))}\sum_{L \in \QSST(\lambda/\mu)} \frac{1}{{|\Rect(L)|}^{\bm{v}_L}}\\
&=\sum_{\Sym(SB(\bm{v}))}\sum_{L \in \QSST(\lambda/\mu)} \frac{1}{{|\Rect(L)|}^{\bm{u}_{\shape(\Rect(L))}(\bm{v})}}\\
&=\sum_{\Sym(SB(\bm{v}))}\sum_{\nu \in \mathcal{G}^Q(\lambda/\mu)}\sum_{M \in \QSST(\nu)} f_{\mu\nu}^\lambda \frac{1}{|M|^{\bm{u}_\nu(\bm{v})}}\\
&=\sum_{\Sym(SB(\bm{v}))}\sum_{\nu \in \mathcal{G}^Q(\lambda/\mu)} f_{\mu\nu}^\lambda \zeta_\nu^Q(\bm{u}_\nu(\bm{v}))
\end{align*}
where the fourth equality follows from the Theorem \ref{LRrule2}.
Thus, we obtain Theorem \ref{skew}.
\end{proof}

\begin{ex}
For $\lambda = (6,3,1), \mu = (2,1)$, and $\bm{v} \in ST(\lambda/\mu, \C)$, the following equality is derived from Theorem \ref{skew}. 
For simplicity, we replace all variables in $SB(\bm{v})$ with the same variables $v_0$ and replace the double subscripts of other variables $v_{ij}$ with single subscripts.

\ytableausetup{boxsize=1.4em}
\begin{align*}
&\zeta^Q_{\lambda/\mu}\left(\vcenter{\hbox{
~\begin{ytableau}
*(black!40){} & *(black!40){} & v_0 & v_0 & v_0 & v_1 \\
\none & *(black!40){} & v_0 & v_0 \\
\none & \none & v_0
\end{ytableau}~}}\right)\\
&=
\zeta^Q_{(4,2,1)}\left(\vcenter{\hbox{
~\begin{ytableau}
v_0 & v_0 & v_0 & v_1 \\
\none & v_0 & v_0 \\
\none & \none & v_0
\end{ytableau}~}}\right)
+
\zeta^Q_{(4,3)}\left(\vcenter{\hbox{
~\begin{ytableau}
v_0 & v_0 & v_0 & v_1 \\
\none & v_0 & v_0 & v_0
\end{ytableau}~}}\right)
+
2\ \zeta^Q_{(5,2)}\left(\vcenter{\hbox{
~\begin{ytableau}
v_0 & v_0 & v_0 & v_0 & v_1 \\
\none & v_0 & v_0 
\end{ytableau}~}}\right)
\end{align*}

since 
$f_{\mu(4,2,1)}^\lambda 
= f_{\mu(4,3)}^\lambda  = 1$, 
$f_{\mu(5,2)}^\lambda = 2$ and $f_{\mu\nu}^\lambda = 0$ for the other partitions $\nu$.

\end{ex}

\section*{Acknowledgement}
The author would like to express her sincere gratitude to her supervisor, Professor Yasuo Ohno, for introducing her to multiple zeta values and giving her his invaluable guidance on both the subject and the attitude towards research. The author is indebted to Professor Wataru Takeda for introducing her to Schur $P$-, $Q$-multiple zeta functions, motivating this research and giving helpful comments from the initial stage of this work. The author is also grateful to Professor Maki Nakasuji for giving her valuable advice. 
The author was partly supported by AIE-WISE Program for AI Electronics by Tohoku University.

\vspace{3mm}

\small
\noindent
HIKARI HANAKI \\
Mathematical Institute, Tohoku University, \\
6-3 Aramaki Aza-Aoba, Aoba-ku, Sendai, 980-8578, Japan.\\
e-mail: hanaki.hikari.s8@dc.tohoku.ac.jp

\end{document}